\documentclass{article}     

\usepackage[a4paper,outer=3cm]{geometry}
\usepackage{fullpage}
\DeclareOldFontCommand{\bf}{\normalfont\bfseries}{\mathbf}

\usepackage[utf8]{inputenc}
\usepackage[T1]{fontenc}
\usepackage{times}

\usepackage{xspace}

\usepackage{longtable}
\usepackage{amsmath,amsthm,amssymb,amsthm, stmaryrd}
\usepackage{enumerate}
\usepackage{multirow}

\usepackage{hyperref}

\usepackage{pgf,tikz}
\usepackage{svg}
\usetikzlibrary{arrows}
\usetikzlibrary{decorations.markings}

\usepackage[nothm]{thmbox}
\usepackage{amsthm}
\usepackage{thmtools}
\usetikzlibrary{patterns}
\usetikzlibrary{calc}

\declaretheorem[parent=section,thmbox=M]{theorem}

\declaretheorem[numberlike=theorem,thmbox=M]{proposition}

\declaretheorem[numberlike=theorem,thmbox=M]{conjecture}

\declaretheorem[numberlike=theorem,thmbox=M]{observation}\declaretheorem[numberlike=theorem,thmbox=M]{property}
\declaretheorem[numberlike=theorem,thmbox=M]{definition}

\newtheorem{claim}{Claim}[theorem]

\declaretheorem[numberlike=theorem]{lemma}
\declaretheorem[numberlike=theorem]{remark}
\declaretheorem[numberlike=theorem]{question}

\newcommand{\Gya}{Gy\'arf\'as\xspace}

\newcommand{\Erd}{Erd\H os\xspace}

\newcommand{\Haj}{Hajnal\xspace}

\newcommand{\bfN}{\mathbf{N}}

\newcommand{\mcal}[1]{\mathcal{#1}}

\newcommand{\mc}{\mathcal}
\newcommand{\mf}{\mathfrak}

\newcommand{\mT}{\mathcal{T}}

\newcommand{\mC}{\mathcal{C}}

\newcommand{\ra}{\rightarrow}
\newcommand{\Ra}{\Rightarrow}

\newcommand{\olra}{\overleftrightarrow}
\newcommand{\ora}[1]{\overrightarrow{#1}}

\DeclareMathOperator{\dic}{\ora \chi}
\DeclareMathOperator{\diomega}{\ora \omega}
\DeclareMathOperator{\dom}{dom}

\newcommand{\Ct}{\ora{C_{3}}}

\newcommand{\TP}{T\hspace{-0.1em}P}

\newcommand{\F}{Forb}

\newenvironment{proofclaim}
	{\noindent {\bf Proof of Claim
	:}}
	{\hfill $\square$ \par\vspace{11pt}}

\usepackage{authblk}

\title{Clique number of tournaments}
\author[1]{Pierre Aboulker}
\author[1,2,3]{Guillaume Aubian}
\author[2]{Pierre Charbit}
\author[1]{Raul Lopes}

\affil[1]{DIENS, \'Ecole normale sup\'erieure, CNRS, PSL University, Paris, France.}
\affil[2]{Université de Paris Cité, CNRS, IRIF, F-75013, Paris, France.}
\affil[3]{Charles University, Prague, Czech Republic.}

\begin{document}

\maketitle

\begin{abstract}
Given a digraph $D$ together with an ordering $\prec$ of its vertices, the \emph{backedge graph} of $D$ with respect to $\prec$ is the undirected graph $D^{\prec}$ with the same vertex set as $D$, where $xy \in E(D^{\prec})$ if $xy \in A(D)$ and $y \prec x$.  
We introduce the notion of the \emph{clique number of a digraph} $D$, defined as the minimum clique number over all backedge graphs of $D$.  
We investigate its relationship with the dichromatic number.  
In particular, this concept allows us to define $\dic$-bounded classes of digraphs, which constitute the main topic of this paper, with a primary focus on tournaments. 
A class of tournaments is $\dic$-bounded if, for every tournament in the class, its dichromatic number is bounded by a function of its clique number. 
We study for which tournaments $H$ the class of $H$-free tournaments is $\dic$-bounded, and prove in particular that $H$ must have a backedge graph that is a forest.  
We prove that if a class of tournaments is $\dic$-bounded, then so is its closure under substitution.  
We also explore the relationship between $\dic$-bounded classes of tournaments and certain conjectures on tournaments. 
We prove that a $\dic$-bounded class of tournaments satisfies the $BIG \Rightarrow BIG$ Conjecture, and that a polynomially $\dic$-bounded class of tournaments satisfies the (tournament) Erd\H{o}s-Hajnal Conjecture.

\end{abstract}

\tableofcontents


\section{Introduction}

In this paper, we only consider \emph{graphs} or \emph{directed graphs} (\emph{digraphs} in short) with no loops, no parallel edges or arcs nor anti-parallel arcs (in particular our digraphs contain no cycle of length $2$).

Given an undirected graph $G$, we denote by $\omega(G)$ the size of a maximum clique of $G$ and by $\chi(G)$ its chromatic number.
Given a digraph $D$, we denote by $\dic(D)$ its {\em dichromatic number}, that is the minimum integer $k$ such that the set of vertices of $D$ can be partitioned into $k$ acyclic subdigraphs. 

The dichromatic number was first introduced by Neumann-Lara~\cite{NL82} in 1982 and  rediscovered by Mohar~\cite{M03} 20 years later.
It is easy to see that for any undirected graph $G$, the \textit{symmetric digraph} $\olra G$ obtained from $G$ by replacing each edge by a digon satisfies $\chi(G) = \dic(\olra G)$. This simple fact permits generalizing results on the chromatic number of undirected graphs to digraphs via the dichromatic number. Such results have (recently) been found in various areas of graph colouring such as extremal graph theory~\cite{BBSS20, HK15, KS20}, algebraic graph theory~\cite{M10}, substructure forced by large dichromatic number~\cite{AAC21, ACN21, ACL19, hero,GSS20, HLNT19, S21}, list dichromatic number~\cite{BHL18, HM11}, dicolouring digraphs on surfaces~\cite{AHKR21, LM17, S19}, flow theory~\cite{H17, KV12}, links between dichromatic number and girth~\cite{HM12, S20}, complexity~\cite{HLM21}. 

Relations between the chromatic number and the clique number  of a graph  have been studied for decades in structural graph theory. The goal of this paper is to introduce a notion of clique number for digraphs, that would generalise the clique number of undirected graphs, and start to investigate its relation with the dichromatic number.  \smallskip 

Before introducing our notion we point out that a natural way to define the clique number of a digraph $D$ is to set it as the size of a maximum symmetric clique of $D$. With this definition, it is clear that the clique number of a symmetric digraph $\olra G$ is the same as the  clique number of $G$. This notion has been investigated in~\cite{perfect_digraphs}, where a generalisation of perfect graphs is considered. Anyway, this definition has the caveat that any oriented graph has clique number one. \smallskip 

Given a digraph $D$, and a total ordering $\prec$ on $V(D)$, we denote $D^{\prec}$ the (undirected) graph   with vertex set $V(D)$ and edge $uv$ if $u\prec v$ and $vu\in A(D)$. 
From now on, whenever we refer to an ordering of the vertices of a (di)graph, we will assume it to be a total ordering and will simply refer to it as an ordering. 
We call $D^{\prec}$  the \emph{backedge graph} of $D$ with respect to $\prec$. It is straightforward that every independent set  of $D^{\prec}$ induces an acyclic subdigraph of $D$. As a consequence, we have that $\dic(D) \le \chi(D^{\prec})$.
Conversely, by taking an ordering built from a $\dic(D)$-dicolouring,
that is taking colour classes one after the other, and ordering each colour class in a topological ordering, we get that:
\[
\dic(D) = \min \, \big\{ \chi(D^{\prec}) : \mbox{$\prec$ is an ordering of $V(D)$} \big\}
\]

This gives an alternative definition for the dichromatic number, which naturally leads  to the following definition of the \emph{clique number of a digraph}\footnote{This definition was introduced during a discussion between the authors and St\'ephan Thomass\'e in Sète during the \href{https://project.inria.fr/anrdigraphs/meetings/}{fifth ANR Digraphs meeting}.} 
\[
\diomega(D) = \min \,\big\{ \omega(D^{\prec}) : \mbox{$\prec$ is an ordering on $V(D)$} \big\}
\]
Observe that, given an undirected graph $G$, $\olra G$ has a unique backedge graph, namely $G$, and thus $\diomega(\olra G) = \omega(G)$. So this definition generalises clique number of graphs to digraphs.

Between the time we submitted the paper and its publication, we discovered that the definition of $\diomega$ had already been introduced in Ilhee Kim's thesis~\cite{K13} (see Chapter 5), where he proves, in particular, the first output of Theorem~\ref{thm:stable_by_subst} (see Theorem 5.4.5 in the thesis). His proof is quite different from ours and a bit longer.
We also point out that, in~\cite{NSS23}, Nguyen, Scott, and Seymour study (among other things) the clique number of backedge graphs of tournaments, without formally defining any particular parameter such as $\diomega$, suggesting anyway that the idea was  on their minds.

Obviously, since $\omega(G)\leq \chi(G)$ for any graph $G$, we also have $\diomega(D)\leq \dic(D)$ for any digraph $D$.
In the context of graphs, since there are families of graphs with clique number $2$ but an arbitrarily large chromatic number, there has been in the past decades a very important amount of work dedicated to the study of so-called $\chi$-bounded classes of graphs, that is classes for which $\chi$ is bounded above by a function of $\omega$. See~\cite{SS20} for a survey on $\chi$-boundedness.
Analogously, we say that a class of digraphs is \emph{$\dic$-bounded} if there exists an increasing function $f$ such that for every digraph $D \in \mathcal C$, $\dic(D) \leq f(\diomega(D))$. The object of the paper is to give first results and conjectures about clique number of tournaments and $\dic$-bounded classes of tournaments. We briefly discuss clique number of general digraphs in Section~\ref{sec:conclusion}. 
\medskip

Note that another definition of $\dic$-boundedness is given in~\cite{ACN21} where the clique number of a digraph $D$ is defined as the maximum size of a transitive tournament contained in $D$\footnote{More precisely, it is defined as the size of a maximum clique of the underlying of $D$, but since any orientation of the complete graphs on $2^k$ vertices contains a transitive tournament on $k$ vertices, if a class of oriented graphs is $\dic$-bounded for one notion, it is also $\dic$-bounded for the other.}. 
Such a definition does not give a lower bound on the dichromatic number, which is the reason why we were looking for another definition.  Observe that for every digraph $D$, $\diomega(D)$ is at most the size of a largest transitive tournament of $D$, so being $\dic$-bounded using this former  definition of clique number implies being $\dic$-bounded using our new notion. 

\medskip

The next section is devoted to notations and definitions used throughout the paper. Section~\ref{sec:chibound} establishes some first properties of clique number, and explores the connections between the clique number of a tournament and the clique number of its backedge graphs. We then endeavour to extend standard results on $\chi$-boundedness to tournaments. In subsection~\ref{sec:subst_S_k}, we describe a simple family having arbitrarily large clique number and prove that $\dic$-boundedness of tournaments is preserved by substitution (and a similar result for some classes of digraphs). We then discuss in subsection~\ref{sec:twinwidth} if, as in the case of classes of undirected graphs with bounded twin-width, classes of tournaments with bounded twin-width are $\dic$-bounded.

A fruitful discussion when studying tournaments involves examining the class of tournaments not containing a given tournament $T$, and deciding which $T$ will ensure that this class has a given property. For example, choices of $T$ guaranteeing a small dichromatic number~\cite{hero} (such $T$ are called heroes), a small domination number~\cite{CKLST18} or a small twin-width~\cite{GT22}  have been studied before. Section~\ref{sec:Hfree} is devoted to this for the property of being $\dic$-bounded. 
In Section~\ref{sec:gentlemen_heroes} we study gentlemen, which are tournaments such that the clique number of tournaments not containing them is bounded, and prove that gentlemen are the same as heroes. In Subsection~\ref{sec:gyarfas_sumner}, we propose an analogue of \Gya-Summner Conjecture for tournaments, proving multiple results supporting this conjecture. We then link $\dic$-binding tournaments to the famous \Erd-Hajnal property and to the $BIG \Ra BIG$ conjecture in Subsection~\ref{sec:erdos_hajnal}. 

Eventually, in Section~\ref{sec:cluster}, we discuss local to global results for clique number, trying to adapt and generalize results of Harutyunyan, Le, Thomassé and Wu in \cite{HLTW19} about dichromatic number and domination number.



\section{Definitions and Notations}


Definitions and notations of this paper that are not explained in this section follow from classical textbooks such as \cite{BG18}, \cite{BM08} or \cite{D05}.

A {\bf $k$-dicolouring} of a digraph is a partition of its vertex set into $k$ subsets inducing acyclic subdigraphs. Alternatively, it is a function  $\phi:V(D) \rightarrow [k]$ such that $D[\phi^{-1}(c)]$ is acyclic for every colour $c\in [k]$.
A digraph is {\bf $k$-dicolourable} if it has a $k$-dicolouring.  
The {\bf dichromatic number} of a digraph $D$, denoted by $\vec{\chi}(D)$, is the least integer $k$ such that $D$ is $k$-dicolourable. 

If $xy\in A(D)$, we say that $x$ is an in \emph{inneighbour} of $x$ and $y$ an \emph{outneighbour} of $x$, and that $x$ and $y$ are \emph{adjacent}. The set of outneighbours (resp. inneighbours) of  $x$ is denoted by $N^{+}(x)$ (respectively $N^{-}(x)$) and we denote by $N(x)$ the set $N^{+}(x)\cup N^{-}(x)$. For a set of vertices $X$, $N^{+}(X)=\cup_{x\in X}N^{+}(x)\setminus X$, $N^{-}(X)=\cup_{x\in X}N^{-}(x)\setminus X$, $N(X)=N^{+}(X)\cup N^{-}(X)$, $N^+[X] = N^+(X) \cup X$ and $N^-[X] = N^-(X) \cup X$.

Given two disjoint sets of vertices $X, Y$ of a digraph $D$, we write $X \Rightarrow Y$ to say that for every $x \in X$ and for every $y \in Y$, $xy \in A(D)$, and we write $X \rightarrow Y$ to say that every arc with one end in $X$ and the other one in $Y$ is oriented from $X$ to $Y$ (but some vertices of $X$ might be non-adjacent to some vertices of $Y$). When $X=\{x\}$ we write $x \Rightarrow Y$ and $x  \rightarrow Y$. 

We also use the symbol $\Ra$ to denote a composition  operation on digraphs: for two digraphs $D_1$ and $D_2$, $D_1\Ra D_2$ is the digraph obtained from the disjoint union of $D_1$ and $D_2$ by adding all arcs from $V(D_1)$ to $V(D_2)$.

A \emph{dominating set} of a digraph $D$ is a set of vertices $X$ such that $N^+[X] = V(D)$. The \emph{dominating number} $\dom(D)$ of $D$ is the size of a smallest dominating set of $D$.

A  \emph{tournament} is an orientation of a complete graph. 
A {\em transitive tournament} is an acyclic tournament and we denote by $TT_n$ the unique acyclic tournament on $n$ vertices.  

Given three tournaments $T_1,T_2,T_3$, we denote by $\Delta(T_1,T_2,T_3)$ the tournament obtained from  disjoint copies of $T_1,T_2,T_3$ by adding arcs in such a way that $T_1\Ra T_2$, $T_2\Ra T_3$ and $T_3\Ra T_1$. If one or more of the tournaments $T_i$ is a transitive tournament $TT_k$, we  simplify the notation by using its size $k$ instead of writing $TT_k$ in the $\Delta$ construction: for example, $\Delta(1,k, T)$ corresponds to $\Delta(TT_1,TT_k,T)$ and $\Delta(1,1,1)$ is simply the directed triangle, which we  also denote by $\vec C_3$.

A {\em class} of graphs (resp. digraphs) is a set of graphs (resp. digraphs) that is closed under induced subgraphs, meaning that if $G$ belongs to the set, then any induced subgraph of $G$ also belongs to the set. Given a set of (di)graphs $\mC$ the {\em hereditary closure} of $\mC$ is the class of all induced sub(di)graphs of elements of $\mC$.

Given a set of tournaments $\mc F$, we say that a tournament $T$ is \emph{$\mc F$-free} if it contains no member of $\mc F$ as a subtournament. We denote by  $\F(\mc F)$ the class of $\mc F$-free tournaments. We write $\F(F_1, \dots, F_k)$ instead of $\F(\{F_1, \dots, F_k\})$ for simplicity. 
\medskip

Because of the definition of $\diomega$, this paper is very often  concerned with orderings on the vertices of a (di)graph. For a (di)graph $G$ we  denote by $\mf S(G)$ the set of orderings of $V(G)$.  
Given a (di)graph $G$ with an ordering $\prec \in \mf S(G)$ and two disjoint subsets $A, B$ of $V(G)$, we write $A \prec B$ to say that for every $a \in A$ and every $b \in B$, $a\prec b$.  
Given a (di)graph $G$ and an  ordering $\prec \in \mf S(G)$, we denote by $(G, \prec)$ the \emph{ordered (di)graph} with ordering $\prec$. 
For an ordered digraph $(D, \prec)$,  an arc $uv$ such that $u\prec v$ is called {\em forward}, and otherwise it is called {\em backward}. 
Recall that given a digraph $D$, and an ordering $\prec$ on $V(D)$, the backedge graph $D^{\prec}$ of $D$ with respect to $\prec$ is the (undirected) graph with vertex set $V(D)$ and edges $uv$ if $u\prec v$ and $vu\in A(D)$ (i.e. $vu$ is a backward arc of $(D, \prec)$). 

An ordering $\prec$ such that $\omega(D^{\prec})= \diomega(D)$ (resp. $\dic(D^{\prec}) = \dic(D)$) is called an \emph{$\diomega$-ordering} of $D$ (resp. $\dic$-ordering). We denote by $\mf S_{\diomega}(D)$ the set of $\diomega$-orderings and by $\mf S_{\dic}(D)$ the set of $\dic$-orderings.


\section{$\chi$-bounded classes of tournaments}\label{sec:chibound}

\subsection{First properties about $\dic$ and $\diomega$ }\label{sec:first_prop}

We begin with an easy fact relating the clique number of a digraph  and the clique number of its strong components.

\begin{property}
The clique number of a digraph is equal to the maximum clique number of its strong components.
\end{property}
\begin{proof}
Assume $D$ has $k$ strong components $C_1, \dots, C_k$ and assume without loss of generality that, for $1 \leq i<j\leq k$, $C_i \ra C_j$. 
It is clear that $\diomega(D) \geq \max \{\diomega(C_i) \mid i=1, \dots, k\}$.
Let $\prec_i$ be an $\diomega$-ordering of $C_i$ for $i=1, \dots, k$. Then, the ordering $\prec$ of $V(D)$, which first lists the vertices of $C_1$, then $C_2$, and so on until $C_k$, and for which the ordering of $V(C_i)$ is the same as $\prec_i$ for each $i=1, \dots, k$, satisfies $\omega(D^{\prec}) = \max \{\diomega(C_i) \mid i=1, \dots, k\}$. 
\end{proof}

As observed in the introduction, the definition of $\diomega$ immediately implies that  $\diomega(T) \leq \dic(T)$ for any tournament $T$. The following proves a relation with the domination number. In Section \ref{sec:cluster} we will expose some other results linking $\dom$, $\dic$ and $\diomega$

\begin{property}\label{prop:dom<omega}
For every tournament $T$, $\dom(T) \leq \diomega(T) \leq \dic(T)$. 
\end{property}

\begin{proof}
    We already know that the second inequality holds. 
    Let $T$ be a tournament, set $V(T) = \{v_1, \dots, v_n\}$ and assume that $v_1 \prec v_2 \prec \dots \prec v_n$ is an $\diomega$ ordering of $T$. 
    Greedily construct a dominating set $X$ of $T$ as follows: $v_1 \in X$, and if $v_i$ is the last vertex added to $X$, add $v_j$ to $X$ where $j$ is minimum such that there is an arc from $v_j$ to every vertex of $X$. 
    Then $X$ is  a dominating set of $T$, and it induces a clique in the backedge graph of $T$ defined by an $\diomega$-ordering. So $\dom(T) \leq \diomega(T)$. 
\end{proof}

In~\cite{NSS23}, the following fundamental inequality is proved (the second inequality is trivial by definition). We give a proof anyway, to make the paper self-contained and familiarize the reader with the notations. Moreover our proof is presented slightly differently than the one in~\cite{NSS23}. 

\begin{theorem}[\cite{NSS23}]\label{eq:bounddic}
For any tournament $T$ and ordering $\prec$ of $V(T)$.
\[ 
\frac{\chi(T^{\prec})}{\omega(T^{\prec})} \leq \dic(T) \leq \chi(T^{\prec})
\]
\end{theorem}

\begin{proof}
Let $T$ be a tournament and $\prec$ an ordering of $V(T)$. 
Let $X \subseteq V(T)$ such that $T[X]$ is a transitive tournament. 
Let us prove that $\chi(T^{\prec}[X]) \leq \omega(T^{\prec})$.

Let $\varphi : X \to \mathbb{N}$ be such that $\varphi(x)$ is the number of vertices of a longest $\prec$-decreasing path in $T^{\prec}[X]$ finishing in $x$. 
We claim that $\varphi$ is a $\omega(T^{\prec})$-colouring of $T^{\prec}[X]$. 
Let $u,v \in X$ with $u\prec v$ and $uv \in E(T^{\prec})$. Then $\varphi(u) \geq \varphi(v) + 1$, so $\varphi$ is a colouring of $T^{\prec}[X]$. 
If $x_1,x_2,x_3 \in X$, $x_3 \prec x_2 \prec x_1$ and $x_1x_2, x_2x_3 \in E(T^{\prec})$,  then $x_1x_2, x_2x_3 \in A(T)$ and thus, since $T[X]$ is a transitive tournament, $x_1x_3 \in A(T)$, i.e. $x_1x_3 \in E(T^{\prec})$. This implies that the vertices of a $\prec$-decreasing path in $T^{\prec}[X]$ induce a clique of $T^{\prec}$. So for every vertex $x \in X$, $\varphi(x) \leq \omega(T^{\prec})$. 

Now, $T$ can be partitioned into $\dic(T)$ transitive tournament, and for each of these transitive tournament $X$, $T^{\prec}[X]$ can be coloured with $\omega(T^{\prec})$ colours. We thus have  $\chi(T^{\prec}) \leq \omega(T^{\prec}) \dic(T)$. 
\end{proof}

Observe that for an arbitrary ordering $\prec$, $\omega(T^{\prec})$ and $\chi(T^{\prec})$ can be arbitrarily larger than $\diomega(T)$ or $\dic(T)$. For example, there is an ordering $\prec$  of $TT_n$ such that $\omega(T^{\prec}) = \chi(T^{\prec})=n$, while $\diomega(TT_n) = \dic(TT_n) = 1$. 
However, an $\diomega$-ordering always provides a good approximation of $\dic$ in the following sense:  
\begin{property}
    For every tournament $T$ and every $\diomega$-ordering $\prec$  we have:
    \[
    \dic(T) \leq \chi(T^{\prec}) \leq \dic(T)^2
    \]
\end{property}

\begin{proof}
Let $T$ be a tournament and $\prec$ an $\diomega$-ordering of $T$. 
By Theorem~\ref{eq:bounddic}, we have that $\chi(T^\prec) \leq \omega(T^\prec) \dic(T)$. 
But since $\omega(T^\prec) = \diomega(T)$ and $\diomega(T) \leq \dic(T)$, we get that:

$$\dic(T) \leq \chi(T^{\prec}) \leq \dic(T)^2$$
\end{proof}

It is natural to ask if the following stronger form of the above property holds (we have  no reason to believe it does, but we could not find a counter-example). 

\begin{question}
Is it true that for every tournament $T$, there exists $\prec \in \mf S(T)$ such that $\prec$ is both an $\diomega$-ordering and a $\dic$-ordering?
\end{question}

Given a class of tournaments $\mT$, let us denote by $\mT^\prec$ the class of all backedge graphs for all orders of tournaments in $\mT$: 
$$
\mc T^{\prec} = \{T^{\prec}: T \in \mc T, \prec \in \mf S(T)\} 
$$
A natural question is whether $\mT$ being $\dic$-bounded has to do with the fact that $\mT^\prec$ is $\chi$-bounded in the usual sense for undirected graphs. And one can ask the same question for the  more restricted class of "optimal" backedge graphs $\mc T^{\prec_{\diomega}}$ : 

$$
\mc T^{\prec_{\diomega}} = \{T^{\prec}: T \in \mc T, \prec \in \mf S_{\diomega}(T)\}
$$

The following theorem answers these questions.

\begin{theorem}\label{thm2:equivchi-boundedness}
    Let $\mc T$ be a class of tournaments. The following properties are equivalent:
    \begin{enumerate}
        \item[(i)] $\mc T$ is $\dic$-bounded.
        \item[(ii)] $\mc T^{\prec}$ is $\chi$-bounded. 
        \item[(iii)] $\mc T^{\prec_{\diomega}}$ is $\chi$-bounded. 
    \end{enumerate}
\end{theorem}

\begin{proof}
   $(i) \Ra (ii)$: let $f$ be a function such that any tournament $T \in \mc T$ satisfies  $\dic(T) \leq f(\diomega(T))$. Now for any tournament $T \in \mc T$ and $\prec \in \mf S(T)$.
    \begin{align*}
        \chi(T^{\prec}) &\leq  \omega(T^{\prec}) \dic(T)& \text {by~Theorem \ref{eq:bounddic}}\\
                        &\leq  \omega(T^{\prec}) f(\diomega(T))& \text{ by $(i)$}\\
                        &\leq  \omega(T^{\prec}) f(\omega(T^{\prec}))&
    \end{align*}

$(ii) \Ra (iii)$ is trivial since $\mc T^{\prec_{\diomega}}\subset \mc T^{\prec}$
\smallskip 

$(iii) \Ra (i)$: let $g$ be a function such that for every $T \in \mc T$ and for every  $\prec \in \mf S_{\diomega}(T)$,  $\chi(T^{\prec}) \leq g(\omega(T^{\prec}))$. Now for any $T \in \mc T$ and $\prec \in \mf S_{\diomega}(T)$. 

\[
    \dic(T)  \leq \chi(T^{\prec})  \leq g(\omega(T^{\prec}))  = g(\diomega(T)) \\
\]
    
\end{proof}

In a converse manner, given a class of undirected graphs $\mathcal C$, one can define $\mc T[\mc C]$ as the set of tournaments $T$ that admits a backedge graph belonging to $\mathcal C$. 
This operation has already been used to construct classes of tournaments with nice properties, see for example~\cite{NSS23}. 
More formally:  

\[
\mc T[\mc C] = \{T: T \mbox{ is a tournament such that there exists} \prec \in \mf S(T) \mbox{ with } T^{\prec} \in \mc C  \}
\]

 
It is natural to wonder if, given a $\chi$-bounded class of undirected graph $\mc C$, the class of tournaments $\mathcal T[\mc C]$ is $\dic$-bounded. It turns out that the answer is no as we explain now. 

Given a graph $G$, denote by $G^c$ be the complement of $G$, that is the graph $\big(V(G), \binom{V(G)}{2} \setminus E(G)\big)$, 
and set $\mc C^{c} = \{G^c: G \in \mc C\}$. 

\begin{property}
    Let $\mc C$ be a class of undirected graphs. Then $\mc C^c \subseteq \mc T[\mc C]^{\prec}$. 
\end{property}
\begin{proof}
  Let $G \in \mathcal \mathcal C$. Let $T$ and $\prec$ be such that $T^{\prec} =G$. Denoting $\prec_r$ the reverse ordering of $\prec$, we get that $T^{\prec_r} = G^c \in \mathcal T[\mathcal C]^{\prec}$. Hence, $\mathcal C^c \subseteq \mc T[\mc C]$.    
\end{proof}

 Let now $\mc C$ be the class of graphs with stability number at most $2$. Then for every $G \in \mathcal C$, $\chi(G) \leq |V(G)| \leq \omega(G) + \omega(G)^2$ (by Kim's classic result in Ramsey theory~\cite{Kim}), so $\mathcal C$ is $\chi$-bounded. But $\mc T[\mc C]^{\prec}$ contains $\mathcal C^c$, the class of triangle-free graphs and is thus not $\chi$-bounded. By Theorem~\ref{thm2:equivchi-boundedness}, this implies that $\mc T[\mc C]^\prec$ is not $\dic$-bounded.



\subsection{Substitution and a class of tournaments with unbounded $\diomega$}\label{sec:subst_S_k}

We start this section by giving a simple and natural construction of tournaments with arbitrarily large clique number.  
Let $\tilde S_1 = TT_1$ and inductively, for $n \geq 1$, let $\tilde S_n=\Delta(\tilde S_{n-1}, \tilde S_{n-1}, \tilde S_{n-1})$. 

\begin{lemma}\label{lem:Stilde_omega}
    For any integer $n$, $\diomega(\tilde S_n)\geq n$. 
\end{lemma}

\begin{proof} 
It is clear that $\diomega(\tilde S_1) = 1$. Let $n \geq 2$ and let $A,B,C$ be the partition of $V(\tilde S_n)$ such that each set induces a copy of $\tilde S_{n-1}$ and $A\Ra B\Ra C \Ra A$. Let $\prec$ be an $\diomega$-ordering of $\tilde S_n$. 

Let $a$ be the first vertex of $\prec$ and assume without loss of generality that $a \in A$. 
By definition of $\diomega$, there exists a clique in the backedge graph of $T[C]=\tilde S_{n-1}$ with respect to $\prec$ that is of size at least $\diomega(\tilde S_{n-1})$, which, together with vertex $a$, form a clique of $\tilde S_n^{\prec}$ of size  $1+\diomega(\tilde S_{n-1}) \leq \diomega(\tilde S_n)$, the desired bound.
\end{proof}

It is also easy to see that $\dic(\tilde S_n)\leq 2\dic(\tilde S_{n-1})$ since one can colour two copies of $\tilde S_{n-1}$ with $\dic(\tilde S_{n-1})$ colours, and then use $\dic(\tilde S_{n-1})$ new colours for the third copy. 
So  $\dic(S_n)\leq 2^n \leq 2^{\diomega(\tilde S_n)}$ (in fact one can improve this to $3/2$ instead of $2$). 
One can wonder whether this inequality is still true for any $T$ in the hereditary closure of the family $\{\tilde S_n \mid n\in \bfN \}$, that is for $T$ being a subtournament of some $\tilde S_n$. 
It is in fact not obvious that this hereditary closure is even  $\dic$-bounded (a subtournament of $\tilde S_n$ can have clique number much smaller than $n$). 
Our next Theorem \ref{thm:stable_by_subst} implies that $\dic(T)\leq 9^{\diomega(T)}$ for every tournament $T$ in this class. This is due to the fact that this class can also be defined through the operation of {\em substitution}, defined below. 
\smallskip 

Given two digraphs $D_1$ and $D_2$ with disjoint vertex sets, a vertex $u \in V(D_1)$, and a digraph $D$, we say that $D$ is obtained by {\em substituting} $D_2$ for $u$ in $D_1$ provided that the following holds: 
\begin{itemize}
    \item $V(D) = (V(D_1) \setminus u) \cup V(D_2)$,
    \item $D[V(D_1) \setminus u] = D_1 \setminus u$,
    \item $D[V(D_2)] = D_2$, 
    \item for all $v \in  V(D_1) \setminus u$ if $vu \in A(D_1)$ (resp. $uv \in A(D_1)$, resp. $u$ and $v$ are non-adjacent in $D_1$), then $V(H_1) \Ra v$  (resp. $v \Ra V(D_2)$, resp. there is no arcs between $v$ and $V(D_2)$) in $D$. 
\end{itemize}

Given a class of digraphs $\mathcal{C}$ we define $\mathcal{C}^{subst}$ to be the closure of $\mathcal{C}$  under substitution.\\

It is a well-known easy-to-prove fact that, if a class of (undirected) graphs is $\chi$-bounded, so is $\mathcal C^{subst}$. The first item of the following theorem proves the analogue for classes of tournaments. The second item applies to classes of digraphs (instead of tournaments) but only in the case of families of digraphs whose underlying graphs have bounded chromatic number. In that case we get a better $\dic$-binding function.

\begin{theorem}\label{thm:stable_by_subst}
    Let $\mathcal{D}$ be a class of digraphs.
    \begin{enumerate}
    \item 
    If $\mathcal{D}$ is a class of tournaments $\dic$-bounded by a function $f$, then $\mathcal{D}^{subst}$ is $\dic$-bounded by function $g(w) = (3wf(w))^w$.
    \item If there exists $K$ such that for every digraph $D$ in $\mathcal{D}$ the underlying graph of $D$ has chromatic number at most $K$, then $\mathcal{D}^{subst}$ is $\dic$-bounded by function $g(w) = (3K)^w$.
    \end{enumerate}
\end{theorem}

\begin{proof}

    We will prove the two statements simultaneously, as the difference only occurs at one particular point of the proof. We want to prove that for every $D \in \mathcal D^{subst}$, $\dic(D) \leq g(\diomega(D))$. Assume by contradiction that the result does not hold and let $D \in \mathcal D^{subst}$ be a counter-example minimizing $\diomega(D) + |V(D)|$.  We refer to this by saying ``by minimality of $D$''.
    
    Define $w:=\diomega(D)$ and observe that $w\geq 2$ since if $\diomega(D)=1$, then $D$ is acylic, so $\dic(D)=1$ and the result holds. 
      
    Since $D \in \mathcal{D}^{subst}$, $D$ can be constructed from a digraph $X \in \mathcal{D}$, with $|V(X)| \geq 2$,  by substituting each vertex $v \in V(X)$ by a digraph $D_v \in \mathcal{D}^{subst}$. 
    Note that, for a fixed $x\in V(X)$, any two vertices in $V(D_x)$ share the same adjacency relation with each vertex of $V(D) \setminus V(D_x)$. This fact will be often used in the proof. 
    
    For any set of vertices $Y \subseteq V(X)$, we define $D_Y = D[\cup_{y \in Y} V(D_y)]$. 
    We call a vertex $v \in V(X)$ \emph{big} if 
    \[
    \dic(D_v) \geq 2g(w-1)+3
    \]
    and \emph{small} otherwise. We call $B$ the set of big vertices and $S$ the set of small vertices (so $V(X) = B \cup S$). 

    In the rest of the proof, we fix  $\prec$ to be an $\diomega$-ordering of $D$. We are going to construct, in two steps, a $g(w)$-dicolouring of $D$. 
    
     First, we are going to prove the existence of a $g(w)$-dicolouring of $D_S$ that satisfies an additional property that will be useful to extend the dicolouring to $D_B$. 
    
    \begin{claim}\label{clm:g-col}
        There is a $g(w)$-dicolouring of $D_S$  such that: 
    if $uv$ is a backward arc of $(D,\prec)$ such that $u \in D_s$ and $v \in D_{s'}$ for some distinct $s,s' \in S$, then $u$ and $v$ receive different colours.
    \end{claim}
    
    \begin{proofclaim}
        In order to do so, we distinguish between the two  cases of the  theorem's statement. First, for each $s \in S$, we set $\varphi_s$ to be a dicolouring of $D_s$ using at most $2g(w-1)+2$ colours (such a dicolouring exists by definition of $D_s$ when  $s \in S$). These dicolourings will be used in both cases. 
        
        \begin{itemize}  
    \item We first treat case $1$, so $\mathcal D$ is a class of tournaments $\dic$-bounded by function $f$. In this case, we prove that $\chi(D_S^{\prec}) \leq g(w)$, which gives the existence of a $g(w)$-dicolouring of $D_S$ with the desired additional property.
    Since $X \in \mathcal D$ and $\diomega(X) \leq \diomega(D) = w$, we have $\dic(X) \leq f(w)$. In particular, there is a $f(w)$-dicolouring $\varphi$ of $X[S]$.  
    For every $u \in V(D_S)$, we set $\phi(u)=(\varphi(s), \varphi_s(u))$ where $s$ is the vertex of $S$ such that $u \in D_s$. 
    Let us prove that $\phi$ is a dicolouring of $D_S$. Assume for contradiction that it is not, and let $C$ be a smallest monochromatic directed cycle in $D_S$. $C$ is not included in a $D_s$, thanks to $\varphi(s)$. 
    By minimality of $C$, and because $D$ is obtained from $X$ by subsituing some vertices of $X$, no two vertices of $C$ belong to the same $D_s$. This implies that $C$ is not monochromatic thanks to $\varphi$ and prove that $\phi$ is a dicolouring of $D_S$. Hence, 
    $\dic(D_S) \leq f(w)(2g(w-1)+2)$. Now, by~Theorem \ref{eq:bounddic} (recall that in this case $\mathcal D$ is a class of tournaments, thus we can apply ~Theorem \ref{eq:bounddic}), we have:
    \[
    \chi(D_S^{\prec}) \leq \omega(D_S^{\prec}) \dic(D_S) \leq wf(w)(2g(w-1) + 2) \leq 3wf(w)g(w-1) \leq g(w).
    \] 
    
    \item In case $2$, we apply a similar approach, but in this case we cannot use~Theorem \ref{eq:bounddic}. We use the fact that the underlying graph of $X$ is $K$-colourable instead.  
    Le $\varphi$ be a $K$-colouring of the underlying graph of $X[S]$. 
    For every $u \in V(D_S)$, we set $\phi(u)=(\varphi(s), \varphi_s(u))$. 
    This colouring uses at most $K(2g(w-1)+2)\leq g(w)$ colours. 
    Moreover, it satisfies a property stronger than the desired "additional" property: if $uv$ is an arc (no need to assume it is backward) such that   $u\in D_s$ and $v\in D_{s'}$ for some distinct $s,s' \in S$, then  $ss' \in E(X)$ and thus  $\varphi(s)\neq \varphi(s')$,  which implies that $u$ and $v$ receive distinct colours as desired. 
    It now remains to prove that $\phi$ is a dicolouring of $D_S$. Let $C$ be a directed cycle in $D_S$. If $C$ is included in a $D_s$, then it is not monochromatic thanks to $\varphi(s)$. 
    Otherwise, it contains an edge $uv$ such that $x \in D_s$ and $v \in D_{s'}$ for two distinct vertices $s$, $s'$ in $S$, and in this case $\phi(u) \neq \phi(v)$ thanks to $\varphi$. 
    \end{itemize} 
    \end{proofclaim}

    For the second step, we first observe that for each big vertex $b \in B$, the digraph $D_b$ has strictly less vertices than $D$ (because $X$ has at least two vertices), so by minimality of $D$ we have $\dic(D_b)\leq g(\diomega(D_b))\leq g(w)$. Moreover since $b$ is big, $\dic(D_b) > g(w-1)$ , so by minimality of $D$ we also get  $\diomega(D_b) = w$.

    For each $b \in B$, and each $u \in D_b$, we define the two following digraphs: 
    \[
    D_b[\prec u] = D[\{x \in V(D_b): x \prec u\}] \text{ and } D_b[u \prec] = D[\{x \in V(D_b): u \prec x\}]
    \]
    and, since $\dic(D_b) \geq 2g(w-1) + 3$,  there is a vertex $m_b \in V(D_b)$ such that 
    \[
    \dic(D_b[\prec m_b]) \geq g(w-1) + 1 \text{ and }  \dic(D_b[m_b \prec]) \geq g(w-1) + 1
    \]

    Hence, by minimality of $D$, for every $b \in B$: 
    \[
    \diomega(D_b[\prec m_b]) = w \text{ and } \diomega(D_b[m_b \prec]) = w
    \]
    
    \begin{claim}\label{clm:m_b}
        For every $b \in B$, if $m_b$ is incident to a backward arc of $(D, \prec)$, then the other extremity of the arc is in $D_b$.
    \end{claim}
    
    \begin{proofclaim}
        Let $b \in B$. There are two symmetric cases depending on whether $m_b$ is the tail or the head of the backward arc.  Assume  by contradiction that $um_b$ is a backward arc of $(D, \prec)$  where  $u\in D_x$ and $x \in X \setminus \{ b \}$. Then, because $D$ is obtained by substituting some vertices of $X$, the fact that $um_b\in A(D)$ implies that $uv\in A(D)$ for any $v\in V(D_b)$.   
    In particular $uv\in A(D)$ for any $v\in D_b[\prec m_b]$. 
    Since $\diomega( D_b[\prec m_b]) \geq w$, a clique of size $w$ in $D_b[\prec m_b]^{\prec}$ together with $v$ form a clique of size $w+1$ in $D^{\prec}$, a contradiction. 
    The argument is exactly the same if $m_bu$ is an arc with $u\prec m_b$ because $ \diomega(D_b[m_b \prec])$ also equals $w$.
    \end{proofclaim}
    We are now ready to conclude. Consider the dicolouring of $V(D_S)$ given by claim~\ref{clm:g-col}, and extend it to a colouring of $V(D)$ by assigning to the vertices of each $D_b$ for $b$ big a dicoulouring using at most $g(w)$ colours (remember that we showed that $\dic(D_b)\leq g(w)$). 
    
    We claim that this defines a valid dicolouring of $D$.
    Assume by contradiction that there exists a monochromatic directed cycle and let $C$ be a minimal such cycle. 
    Since for a fixed $x\in V(X)$, any two vertices in $V(D_x)$ share the same adjacency relation with each vertex of $V(D) \setminus V(D_x)$, the minimality of $C$ implies that $C$ contains at most one vertex of $V(D_x)$ for any $x \in V(X)$, for otherwise $C$ would be entirely included into some $D_x$, which is not possible since the colouring is a valid dicolouring on each digraph $D_x$. 
    Now define $C'$ to be the directed cycle obtained from $C$ by replacing each vertex belonging to  $V(D_b)$ for some $b \in B$ by the vertex $m_b$. 
    Since $C'$ is a directed cycle, it must contain some backward arc of $(D, \prec)$, i.e. some arc $uv$ with $v\prec u$. 
    Since $u$ and $v$ do not belong to the same $D_x$,  by Claim~\ref{clm:m_b}, both vertices $u$ and $v$ belong to $D_S$. 
    But then they are vertices of $C$ and the arc $uv$ is thus monochromatic and backward, which contradicts the additional property of the dicolouring of $D_S$ established in the first step of the proof.  
    \end{proof}

Note that the hereditary closure of $\{\tilde S_n, n\in\mathbb N\}$ mentioned at the beginning of this subsection is easily seen to be exactly $\{TT_1,TT_2,\Ct\}^{subst}$. Therefore the first item implies $\dic(T)\leq 9^{\diomega(T)}$ for any $T$ which is a subtournament of some $\tilde S_n$. We do not know if this class is $\dic$-bounded by a polynomial function. One can prove for example that the order of magnitude of $\dic(\tilde S_n)$ is $(3/2)^n$ but it could be that the lower bound $\diomega(\tilde S_n)\geq n$ given by Lemma \ref{lem:Stilde_omega} is far from tight.

Chudnovsky, Penev, Scott and Trotignon~\cite{CPST13} proved that if a class of graphs is $\chi$-bounded by a polynomial function, so is its closure under substitution. Could it be that the same holds for tournaments? 

\begin{question}
    Is it true that if a class of tournaments $\mc T$ is polynomially $\dic$-bounded, then so is $\mc T^{subst}$?
\end{question}

Before closing this section on substitutions let us define another sequence of tournaments belonging to $\{TT_1,TT_2,\Ct\}^{subst}$ that will be of use in the proof of Theorem \ref{thm:gentlemen}.  

\begin{definition}\label{def:Sn}
Let $S_n$, $n\geq 1$, denote the sequence of tournaments defined recursively by $S_1 = TT_1$ and, for $n \geq 2$, $S_n=\Delta(1, S_{n-1}, S_{n-1})$.
\end{definition}

It is easy to observe that $\dic(S_n)=n$. Since $S_n$ is obviously a subtournament of $\tilde S_n$, we have $\dic(S_n)\leq 9^{\diomega(S_n)}$, and therefore $\diomega(S_n)\geq \log_9(n)$.
Again it could be that this logarithm is not necessary. It is clear that $\diomega(S_1) = 1$, $\diomega(S_2) = 2$ and it is not  hard (but a bit laborious)  to prove that $\diomega(S_3) = 2$ and $\diomega(S_4) = 3$. The clique number of $S_k$ for $k \geq 5$ is not known but we doubt that one can compute an exact formula for it.

\subsection{Are tournaments with bounded twin-width  $\dic$-bounded?}\label{sec:twinwidth}

\emph{Twin-width} is a parameter introduced in~\cite{tww1} to measure the complexity of binary relational structures. We refer to Sections 2.2 and 2.6 of~\cite{GT22} for the general definitions of twin-width in the context of binary relational structures.

Since this paper focuses on specific structures, we reformulate these definitions for the cases where the structure, denoted by $S$, is a graph, an ordered graph, a tournament or an ordered tournament. 
A \emph{contraction sequence} for  $S$  is a sequence $\mathcal P_n, \mathcal P_{n-1} \dots, \mathcal P_1$ of partitions of the set of vertices of $S$  where $\mathcal P_n$ is the partition into singletons, $\mathcal P_1$ is the partition into a single part, and $\mathcal P_{i-1}$ is obtained by merging two parts of $\mathcal P_i$.  
For a given part $X$ of $\mathcal P_i$, the \emph{error degree} of $X$ in $(S, \mathcal P_i)$ is the number of parts $Y \in \mathcal P_i \setminus X$ that are not homogeneous to $X$. 
Let us detail the meaning of \emph{not being homogeneous} in the cases of graphs, tournaments,  ordered graphs and ordered tournaments.  We say that $X \not\prec Y$ if there is $x \in X$ and $y \in Y$ such that $y \prec x$.
\begin{itemize}
    \item If $S$ is a graph, then both $(X \times Y) \cap E(S)$ and $(X \times Y) \setminus E(S)$ are non-empty.
    \item If $S$ is a tournament, then both $(X \times Y) \cap A(S)$ and $(Y \times X) \cap A(S)$ are non-empty.
    \item If $S = (G, \prec)$ is an ordered graph, then either:
    \begin{itemize}
        \item $(X \times Y) \cap E(G)$ and $(X \times Y) \setminus E(G)$ are both non-empty, or
        \item $X \not\prec Y$ and $Y \not\prec X$.
    \end{itemize}
    \item If $S = (T, \prec)$ is an ordered tournament, then either:
    \begin{itemize}
        \item $(X \times Y) \cap E(S)$ and $(X \times Y) \setminus E(S)$ are both non-empty, or
        \item $X \not\prec Y$ and $Y \not\prec X$.
    \end{itemize} 
\end{itemize}
The \emph{width} of the contraction sequence is the maximum error degree of $(S, \mathcal P_i)$ over all choices of $i \in \{1, \dots, n\}$. Finally, the \emph{twinwidth} of $S$, denoted by $tww(S)$, is the minimum width of a contraction sequence for $S$.  
\medskip

Classes of graphs of bounded twin-width have been shown to be $\chi$-bounded~\cite{tww3,PS22}, and even polynomially $\chi$-bounded~\cite{BT23}. 

\begin{theorem}[\cite{BT23}]\label{thm:poly}
   For every $k \geq 1$, the class of undirected graphs with twin-width at most $k$ is polynomially $\chi$-bounded. 
\end{theorem}

Observe that for every integer $k \geq 2$, $S_k$ (see the end of Subsection~\ref{sec:subst_S_k} for the definition of $S_k$) has dichromatic number $k$ and twin-width $1$.  Indeed, $tww(S_1) = 1$ trivially, and assuming that $tww(S_{k-1}) = 1$, the following contraction sequence proves that $tww(S_k) = 1$: contract each copy of $S_{k-1}$ into a single part, one at a time, going through partitions of error degree at most $1$ (and $0$ when the copy of $S_k$ is reduced to a single vertex), what remains is a directed triangle, for which the contraction sequence is trivial.  
So tournaments with bounded twin-width can have arbitrarily large dichromatic number.  

\begin{conjecture}\label{conj:tww_chi_bounded}
Let $k \geq 1$. The class of tournaments with twin-width at most $k$ is $\dic$-bounded. 
\end{conjecture}

Recall that, given a tournament $T$ and an  ordering $\prec$ of $V(T)$, we denote by $(T, \prec)$ the ordered tournament with ordering $\prec$,  and by $(T^{\prec}, \prec)$ the ordered backedge graph with ordering $\prec$. 

Along with Theorem~\ref{thm:poly}, the following conjecture implies Conjecture~\ref{conj:tww_chi_bounded}.

\begin{conjecture}\label{conj:good_ordering_tww_omega}
    There exists a function $f$, such that for every tournament $T$, there exists an ordering $\prec^*$ of $V(T)$ such that: 
    \[
    \omega(T^{\prec^*}) \leq f( \diomega(T)) \quad \text{ and } \quad 
    tww((T, \prec^*)) \leq f(tww(T))
    \] 
\end{conjecture}

\begin{theorem}
    Conjecture~\ref{conj:good_ordering_tww_omega} implies Conjecture~\ref{conj:tww_chi_bounded}.
\end{theorem}

\begin{proof}
    Let $\mc T_k$ the class of tournaments with twin-width at most $k$. For each $T \in \mc T_k$, we associate an ordering $\prec_T^*$ given by Conjecture~\ref{conj:good_ordering_tww_omega}. For every $T \in \mathcal T_k$, we have: 
    
\begin{equation}\label{eq:tww}
    tww(T^{\prec_T^*}) \leq tww((T^{\prec_T^*}, \prec_T^*)) = tww((T, \prec_T^*)) \leq f(k)
\end{equation}
    The first inequality holds because the definition of homogeneity for ordered graphs is strictly more restrictive than that for graphs.  
The equality holds because the definitions of homogeneity for the ordered graph $(T^{\prec}, \prec)$ and the ordered tournament $(T, \prec)$ coincide, up to replacing, for every pair of vertices $x, y$ with $x \prec y$, each edge $xy \in E(T^{\prec})$ by $xy \in A(T)$, and each non-edge $xy \not\in E(T^\prec)$ by $yx \in A(T)$. 
The last inequality comes from the property of $\prec_T^*$.

Hence, the class of undirected graphs $\mathcal C_k=\{T^{\prec_T^*}\mid T \in \mc T_k\}$ has bounded twin-width, and  is thus $\chi$-bounded by a polynomial function $g$. 

Let $T \in \mc T_k$. We have: $T^{\prec_T^*} \in \mc C_k$, $tww(T^{\prec_T^*}) \leq f(k)$ by~\eqref{eq:tww},  $\omega(T^{\prec_T^*}) \leq f(\diomega(T))$ by the choice of $\prec_T^*$. 
Hence, 
\[
\chi(T^{\prec_T^*}) \leq g(f(\diomega(T)))
\]
and thus $\dic(T) \leq g(f(\diomega(T))$. 
\end{proof}

Geniet and Thomassé~\cite{GT22} introduced a particular ordering of tournaments called \emph{BST-ordering}. Informally, a $BST$-ordering of a tournament $T$ is based on a rooted binary search tree  on vertex set $V(T)$ with the property that for every vertex $x$, the left child of $x$ and its descendants are in $N^-(x)$, and the right child of $x$  and its descendants are in $N^+(x)$, and the ordering is the left-to-right ordering defined by this tree. 
See~\cite{GT22} Section 4 for a formal definition.  
They prove that $BST$-orderings give an approximation of the twin-width in the following sense:
\begin{theorem}[\cite{GT22}]
    There exists a function $f$ such that, for every tournament $T$ and any $BST$-ordering $\prec$ of $T$, we have: 
    
    \[
    tww(T, \prec) \leq f(tww(T))
    \]
\end{theorem}

Hence, $BST$-orderings are natural candidates for the ordering of Conjecture~\ref{conj:good_ordering_tww_omega}. 

\begin{conjecture}
    There exists a function $f$ such that, for every tournament $T$, there exists a $BST$-ordering $\prec$ of $T$ such that:
    \[
    \omega(T^{\prec}) \leq f(\diomega(T))
    \]
\end{conjecture}

\section{Classes of tournaments defined by forbidding a single tournament}\label{sec:Hfree}

In this section, we will investigate the classes  defined by forbidding a single tournament. 
Recall that, given a tournament $H$, $\F(H)$ is the class of  $H$-free tournaments. 
Our main question will be to understand which tournament $H$ are such that $\F(H)$ is $\dic$-bounded. Such a tournament is said to be  \emph{$\dic$-binding}.


\subsection{Gentlemen are the same as heroes}\label{sec:gentlemen_heroes}


The most trivial case of $\chi$-bounding function is a constant function. A tournament $H$ is a \emph{hero} if there exists an integer $c_H$ such that every $H$-free tournament $T$ has dichromatic number at most $c_H$. 

In a seminal paper, Berger, Choromanski, Chudnovsky, Fox, Loebl, Scott, Seymour and Thomassé~\cite{hero} characterized heroes:
\begin{theorem}[Berger, Choromanski, Chudnovsky, Fox, Loebl, Scott, Seymour and Thomassé\ \cite{hero}]\label{thm:heroes}
A tournament $H$ is a hero if and only if:
\begin{itemize}
\item $H=TT_1$, or
\item $H=H_1 \Ra H_2$, where $H_1$ and $H_2$ are heroes in tournaments, or
\item $H=\Delta(1, k, H_1)$ or $H = \Delta(1, H_1, k)$, where $k\geq 1$ and $H_1$ is a hero in tournaments.
\end{itemize}
\end{theorem}

Similarly, we say that a tournament $H$ is a \emph{gentleman} if there exists a number $c_H$ such that every $H$-free tournament has clique number at most $c_H$. Since $\diomega(T)\leq \dic(T)$ for any tournament $T$,  heroes are gentlemen. We prove that the converse is also true.

\medskip 

In \cite{NSS23}, Nguyen, Scott and Seymour introduce a class of tournaments called \emph{crossing tournaments} (it is the class $\mathcal T[\mc C]$ where $\mc C$ is the class of circle graphs). They prove that crossing tournaments are $S_3$-free (see Definition~\ref{def:Sn} for the definition of $S_3$) and that they can have arbitrarily large clique number (see respectively 6.2 and 6.4 in~\cite{NSS23}).  
It is a key ingredient in the proof of the following theorem.

\begin{theorem}\label{thm:gentlemen}
   For any tournament $T$, $T$ is a gentlemen if and only if it is a hero.
\end{theorem}

\begin{proof}
    For every tournament $T$, $\diomega(T) \leq \dic(T)$, thus it is clear that all heroes are gentlemen.

    Let us now prove that all gentlemen are heroes. Suppose there exists a gentleman $H$ that is not a hero, and let it be chosen so as to minimize $|V(H)|$. Since all subtournaments of a gentleman are gentlemen (because tournaments not containing a subtournament of $H$ do not contain $H$ and thus have bounded clique number), every subtournament of $H$ is a hero by minimality of $V(H)$. Consider the sequence of tournaments $S_n$ defined in Definition \ref{def:Sn}. Since this sequence has unbounded clique number and $H$ is a gentlemen, there exists an integer $k$ such that $H$ is a subtournament of $S_k$. This implies that either $H = A \Ra B$ or $H = \Delta(1,A,B)$ for some tournaments $A$ and $B$. Since $A$ and $B$ are two strict subtournaments of $H$, they are heroes by minimality of $H$. Thus $H \neq A \Ra B$ for $H$ would be a hero by Theorem \ref{thm:heroes}. Thus $H = \Delta(1,A,B)$.
    
    But we know that $S_3$ is not a gentleman since crossing tournaments are $S_3$-free and can have arbitrarily large clique number. Thus $H$ does not contain $S_3 = \Delta(1, \vec C_3 , \vec C_3)$. This implies that one of $A$ or $B$ does not contain $\vec C_3$ and thus either $A$ or $B$ is a transitive tournament, which implies $H$ is a hero, a contradiction.
    
\end{proof}


\subsection{\Gya-Sumner Conjecture for tournaments}\label{sec:gyarfas_sumner}


A tournament $H$ is said to be \emph{$\dic$-binding} if $\F(H)$ is $\dic$-bounded. 
We propose the following analogue of the celebrated \Gya-Sumner Conjecture~\cite{Gya87, S81} that states that a graph $F$ is $\chi$-binding if and only if $F$ is a forest (where, as in the directed case, a graph $F$ is $\chi$-binding if the class of graphs not containing $F$ as an induced subgraph is $\chi$-bounded).  

\begin{conjecture}\label{conj:GS_tournaments}
    A tournament $H$ is $\dic$-binding if and only if $H$ has a backedge graph which is a forest.  
\end{conjecture}

\begin{remark}\label{rem:G-S_tournament_false}
After we submitted the paper, the second author disproved Conjecture~\ref{conj:GS_tournaments} \cite{A24}. 
\end{remark}

Note that no analogue of the above conjecture can hold for general digraphs, and in fact, not even for the class of oriented graphs, that is, digraphs with no digon. 
Indeed, for every oriented graph $F$ on at least $3$ vertices, the class $\mathcal D$ of oriented graphs with no induced copy of $F$ is not $\dic$-bounded. If $F$ is not a tournament, then $\mathcal D$ contains all tournaments, and thus is not $\dic$-bounded. And if $F$ is a tournament, then $\mathcal D$ contains all oriented graphs that have no triangle in their underlying graph. Such oriented graphs have clique number at most $2$, and can have arbitrarily large dichromatic number.
\medskip 

Despite the link between $\dic$-bounded classes of tournaments and $\chi$-bounded classes of graphs given by Theorem~\ref{thm2:equivchi-boundedness}, we were not able to prove that \Gya-Sumner Conjecture implies or is implied by Conjecture~\ref{conj:GS_tournaments}. We now believe that the two conjectures are independent, but we would be very happy if a bridge between them was shown. \medskip 

To support the conjecture, we prove  that :
\begin{itemize}
\item the "only if" part is true (Theorem~\ref{thm:only_if_part_GS}), 
\item it is enough to prove it for trees instead of forests (Proposition~\ref{lem:GS_for_trees}), 
\item if it holds for a tournament $T$, then it holds for the tournaments obtained by reversing every arc of $T$ (Proposition~\ref{lem:reverse}),
\item if it holds for two tournaments $H_1$ and $H_2$, then it holds for the tournament $H_1 \Ra H_2$ (Theorem~\ref{thm:GS_H=>H}), 
\end{itemize}

A \emph{star} is a tree that has at most one non-leaf vertex. 
We also prove that  heroes  admit a backedge graph that is a disjoint union of stars. See Proposition~\ref{prop:herostar}. 

At the end of the section we also discuss the case of tournaments $T$ that admit a backedge graph that is a matching.

\begin{theorem}\label{thm:only_if_part_GS}
Let $H$ be a tournament. If $H$ is $\dic$-binding, then $H$ admits an ordering whose backedge graph is a forest. 
\end{theorem}

\begin{proof}
Let $H$ be a tournament that does not admit an ordering whose backedge graph is a forest. 
    Let $\mathcal C$ be the class of undirected graphs with girth at least $|V(H)|+1$, and let $\mathcal T[\mathcal C]$ be the class of tournaments that admit a graph of $\mathcal C$ as a backedge graph. 
    
    Let $T \in \mathcal T[\mathcal C]$ and let $X \subseteq V(T)$ such that $|X| = |V(H)|$. $T$ admits an ordering such that the backedge graph has girth at least $|V(H)|+1$, so $T[X]$ admits an ordering for which the backedge graph is a forest. So $T[X] \neq H$. This proves that tournaments in $\mathcal T[\mc C]$ are $H$-free. 

    Since for every $T \in \mathcal T[\mathcal C]$, a backedge graph of $T$ has girth $|V(H)| + 1 \geq 4$, we have  $\diomega(T) \leq 2$. 
    
    By a celebrated result of Erd\H{o}s~\cite{E59}, for every integer $k$, there exists $G \in \mathcal C$ such that $\chi(G) \geq k$. Let $T \in \mathcal T[\mathcal C]$ such that $T$ admits an ordering $\prec$ such that $T^{\prec}=G$.  By~Theorem \ref{eq:bounddic}, 
    $\dic(T) \geq  \chi(T^{ \prec }) / \omega(T^{\prec}) = k/2$. 

    This proves that there are $H$-free tournaments with clique number $2$ and arbitrarily large dichromatic number, i.e. $H$ is not $\dic$-binding. 
\end{proof}

Let $T$ be a tournament admitting a forest as a backedge graph. We claim that $T$ also admits a tree as a backedge graph. 
Indeed, let $v_1 \prec v_2 \prec \dots \prec v_n$ be an ordering of $V(T)$ such that $T^{\prec}$ is a forest and among such orderings, assume it minimizes the number of connected components of $T^{\prec}$. 
We claim that $T^{\prec}$ is a tree. Assume for contradiction that it is not. Let $v_i$ be the smallest vertex  not in the same connected component of $T^{\prec}$ as $v_1$. Then the backedge graph resulting from switching $v_{i-1}$ and $v_i$ in the ordering is obtained from $T^{\prec}$ by adding the edge $v_{i-1}v_i$ and thus has one less connected component than $T^{\prec}$, a contradiction. We thus have the following:

\begin{proposition}\label{lem:GS_for_trees}
    It is enough to prove Conjecture~\ref{conj:GS_tournaments} for tournaments that admit a tree as a backedge graph. 
\end{proposition}

As explained in the previous subsection, if $H$ is a hero, then $\F(H)$ is $\dic$-bounded by a constant function, which implies by Theorem \ref{thm:only_if_part_GS} that they admit a backedge graph that is a forest. The following proposition proves that for heroes, this backedge graphs is a star forest. 
\begin{proposition}\label{prop:herostar}
    If $H$ is a hero, then $H$ admits a backedge graph that is a disjoint union of stars.
\end{proposition}
\begin{proof}
    We prove this using the inductive construction of heroes given by Theorem \ref{thm:heroes}. It is true for $TT_1$ so we need to maintain this property if $H=H_1\Ra H_2$ and if $H=\Delta(1,k,H_1)$.

    If $H=H_1\Ra H_2$, consider orderings $\prec_i$ of $H_i$ given by the induction, and simply construct the ordering on $V(H)$ in which all vertices of $H_1$ are placed before those of $H_2$ (respecting $\prec_1$ and $\prec_2$). This adds no new back arc, so the back edge graph is the union of those of $H_1$ and $H_2$, so we have our result.

    If $H=\Delta(1,k,H_1)$, then consider the ordering $\prec$ of $H_i$ given by the induction, and construct the ordering on $V(H)$ obtained by placing the vertices of $TT_k$ first so that all arcs go forward, then the vertices of $H_1$ in the ordering $\prec$, and finally the vertex $x$ corresponding to the "1" in $\Delta(1,k,H_1)$. The only new back arcs are the one from $x$ to the vertices of the $TT_k$, which produce a star, so we again get our desired result.
\end{proof}

Given a tournament $T$, the reverse $T_r$ of $T$ is the tournament obtained from $T$ by reversing the direction of every arc. Given an ordering $\prec$ of $T$, we define $\prec_r$ to be the ordering of $T$ obtained by reversing $\prec$.

\begin{proposition}\label{lem:reverse}
    Let $H$ be a tournament. If $H$ is a $\dic$-binding, then the  reverse $H_r$ of $H$ is also $\dic$-binding. 
\end{proposition}

\begin{proof}
    Observe that, for every tournament $H$ and every ordering $\prec$ of $H$,  $H^{\prec} = H_r^{\prec_r}$. So $\F(H)^{\prec} = \F(H_r)^{\prec}$ and the result holds by Theorem~\ref{thm2:equivchi-boundedness}. 
\end{proof}

In order to prove our next theorem, we will use the following very nice result of Le, Harutyunyan, Thomassé and Wu that states that if the outneighbourhood of each vertex of a tournament has bounded dichromatic number, so does the whole tournament. We will discuss more of these kinds of ``local to global'' properties in Section~\ref{sec:cluster}. 

\begin{theorem}[Le, Harutyunyan, Thomassé and Wu\ \cite{HLTW19}]\label{thm:local_to_global}
There exists a function $\lambda_1$ such that, for every integer $t$, if $T$ is a tournament such that for every $v \in V(T)$, $\dic(T[N^+(v)]) \leq t$, then $\dic(T) \leq \lambda_1(t)$. 
\end{theorem}

Following~\cite{KN23}, we define the \emph{neighbourhood $N(uv)$ of an arc $uv$} to be  $N^-(u) \cap N^+(v)$, that is the set of vertices forming a directed triangle with $uv$.

\begin{theorem}\label{thm:GS_H=>H}
    Let $H_1$ and $H_2$ be two tournaments. If $H_1$ and $H_2$ are  $\dic$-bindings, then $H_1 \Ra H_2$ is also $\dic$-binding. 
\end{theorem}

\begin{proof}
    For $i=1,2$, let $f_i$ be a $\dic$-binding function for $\F(H_i)$.  
    Set $f=\max(f_1, f_2)$ and $s=\max(|H_1|, |H_2|)$. 
    We are going to prove that the class of $(H_1 \Ra H_2)$-free tournaments is $\dic$-bounded by a function $g$ that will be fixed during the proof.  

    The result trivially holds for $g(1) = 1$, since tournaments with clique number $1$ are transitive and thus have dichromatic number at most $1$.  
    Let $k \ge 1$. Assume the result holds for $(H_1 \Ra H_2)$-free tournaments with clique number at most $k$, and let $T$ be a $(H_1 \Ra H_2)$-free tournament with clique number $k+1$.


    Let $\prec$ be an $\diomega$-ordering of $T$. 
    We call an arc $uv$ heavy if $\dic(T[N(uv)]) \geq  2g(k) + 1$, and light otherwise.

    \begin{claim}\label{clm:uv_heavy_arc}
        If $uv$ is a heavy arc, then $v \prec u$. 
    \end{claim}

    \begin{proofclaim}
        Assume for contradiction that $u \prec v$. 
        Let $(L, M, R)$ be the partition of $N(uv)$ such that $L \prec u \prec M \prec v \prec R$.
        Since $L \cup M \prec v$ and $v \Ra L \cup M$, $\diomega(T[L \cup M]) \leq k$. Since $u \prec R$ and $R \Ra u$, $\diomega(T[R]) \leq k$. Hence, by induction, $\dic(T[N(uv)]) \leq 2g(k)$, a contradiction with the fact that $uv$ is heavy. 
    \end{proofclaim}
 
    \begin{claim}\label{clm:vertices_are_small}
        Every vertex $u$ satisfies $\min(\dic(T[N^-(u)]), \dic(T[N^+(u)]) \leq f(k+1) + g(k) + 2sg(k)$.
    \end{claim}

    \begin{proofclaim}
        Let $u \in V(T)$. 
        Set $X_h^+ = \{v \mid uv \text{ is heavy}\}$, $X_h^- = \{v \mid vu \text{ is heavy}\}$, $X_{\ell}^+ = \{v \mid uv \text{ is light}\}$ and $X_{\ell}^- = \{v \mid vu \text{ is light}\}$. 
        Note that $N^+(u) = X_h^+ \cup  X_{\ell}^+$ and $N^-(u) = X_h^- \cup  X_{\ell}^-$

        By claim~\ref{clm:uv_heavy_arc}, $X_h^+ \prec u \prec X_h^-$, and since $X_h^- \Ra u \Ra X_h^+$, we have $\diomega(T[X_h^+]) \leq k$ and $\diomega(T[X_h^-]) \leq k$ and thus, by induction, $\dic(T[X_h^+]) \leq g(k)$ and $\dic(T[X_h^-]) \leq g(k)$. 

        Now, if $\min(\dic(T[X_{\ell}^+]), \dic(T[X_{\ell}^-]))\leq f(k+1) + 2sg(k)$, we are done. 

        Assume first that $\dic(T[X_{\ell}^+]) > f(k+1) + 2sg(k)$. 
        Then there exists $V_2 \subseteq X_{\ell}^+$ such that $T[V_2]=H_2$. Let $X \subseteq X_{\ell}^-$ such that $X \Ra V_2$. 
        Then $X$ is $H_1$-free and thus $\dic(T[X]) \leq f(k+1)$. 
        Now, since $X_{\ell}^- \setminus X \subseteq N^-(V_2)$, for each $w \in X_{\ell}^- \setminus X$, there exists $v \in V_2$ such that $w \in N(uv)$. Since for each $v \in V_2$, $uv$ is light, i.e. $\dic(T[N(uv)]) \leq 2g(k)$, we have $\dic(T[X_{\ell}^- \setminus X]) \leq 2sg(k)$. Hence $\dic(T[N^-(u)]) \leq g(k) + f(k+1) + 2sg(k)$. 

        If $\dic(T[X_{\ell}^-]) > f(k+1) + 2sg(k)$, we similarly get that $\dic(T[N^+(u)]) \leq g(k) + f(k+1) + 2sg(k)$. 
     \end{proofclaim}

    By claim~\ref{clm:vertices_are_small}, we can partition $V(T)$ into two sets $V_1$ and $V_2$ such that for every $v \in V_1$, $\dic(T[N^+(v)]) \leq f(k+1) + g(k) + 2sg(k)$ and for every $v \in V_2$, $\dic(T[N^-(v)]) \leq f(k+1) + g(k) + 2sg(k)$. 
    Hence, by Theorem~\ref{thm:local_to_global}, $\dic(V_i) \leq \lambda_1(f(k+1) + g(k) + 2sg(k))$ for $i=1,2$. Hence, the function $g(k+1)=2\lambda_1(f(k+1) + g(k) + 2sg(k))$ is a $\dic$-binding function for the class of $(H_1 \Ra H_2)$-free tournaments. 
\end{proof}

Recall that an ordered graph is a pair $(G, \prec)$ where $G$ is a graph and $\prec \in \mf S(G)$. 
We say that an ordered graph $(H,\prec_H)$ is an induced ordered subgraph of an ordered graph $(G,\prec_G)$ if there exists an injective mapping $\phi:V(H)\to V(G)$ such that for all vertices $x,y$ of $V(H)$, $x \prec_H y$ if and only if $\phi(x) \prec_G \phi(y)$ and $xy\in E(H)$ if an only if $\phi(x)\phi(y)\in E(G)$. 

Given a set $\mcal O$ of ordered graphs, we define $\F_o(\mcal O)$ as the class of ordered graphs that do not contain any member of $\mcal O$ as an induced ordered subgraph. We say that a class of ordered graphs $\mcal O$ is $\chi$-bounded if the set of graphs $G$ such that there exists $\prec$ with $(G,\prec)\in \mcal O$ is $\chi$-bounded (we simply ignore the orderings here).

Let $\mathcal T$ be a class of tournaments. 
Recall that $\mathcal T^{\prec}$ is the set of graphs that are  backedge graphs of a tournament in $\mathcal T$. We now define the ordered version of it as follows: 
\[
\mathcal T_o^{\prec} = \{(T^{\prec}, \prec): T \in \mathcal T, \prec\in \mf S(T)\}
\]
Note that, given a tournament $T$, $\{T\}^{\prec}$ is the set of backedge graphs of $T$ and $\{T\}^{\prec}_o$ the set of ordered backedge graphs of $T$. The following is an ordered analogue of Theorem~\ref{thm2:equivchi-boundedness}. 

\begin{property}\label{prop:tournaments_orderedgraphs}
    Let $F$ be a tournament. The class of tournaments $\F(F)$ is $\dic$-bounded if and only if the class of ordered undirected graphs $\F_o(\{F\}^{\prec}_o)$ is $\chi$-bounded. 
\end{property}

\begin{proof}
    Assume first that the class of tournaments $\F(F)$ is $\dic$-bounded by a function $f$. Let $(G, \prec_G) \in \F_o(\{F\}^{\prec}_o)$. Let $T$ be the tournament on the same vertex set as $G$, such that $(T^{\prec_G}, \prec_G) = (G, \prec_G)$ (i.e. $T$ is the tournament obtained from $(G, \prec_G)$ by orienting edges of $G$ from right to left, and for each pair of non adjacent vertices  $x,y$, we add the arc $xy$ if $x \prec y$, and the arc $yx$ otherwise). 
    Since $(G, \prec_G) \in \F_o(\{F\}^{\prec}_o)$, $T \in \F(F)$ and thus $\dic(T) \leq f(\diomega(T)) \leq f(\omega(G))$ (because $\diomega(T) \leq \omega(T^{\prec_G}) = \omega(G)$). We then have the following, where the first first inequality holds by~Theorem \ref{eq:bounddic}
    \[
    \chi(G) = \chi(T^{\prec_G}) \leq \dic(T) \omega(T^{\prec}) \leq f(\omega(G))\omega(G)
    \]
    which proves that $\F_o(\{F\}^{\prec}_o)$ is $\chi$-bounded. 

    Assume now that the class of ordered graphs $\F_o(\{F\}^{\prec}_o)$ is $\chi$-bounded by a function $f$. 
    Let $T \in \F(F)$ and let $\prec$ be an $\diomega$-ordering of $T$. 
    Then $(T^{\prec}, \prec) \in \F_o(\{F\}^{\prec}_o)$. Hence:
    $\dic(T) \leq \chi(T^{\prec}) \leq f(\omega(T^{\prec})) = f(\diomega(T))$
\end{proof}

Bria\'nski, Davies and Walczak~\cite{BDW22} studied for which ordered graphs $(G, \prec_G)$, $\F_o((G, \prec_G))$ is $\chi$-bounded, and claimed in a personal communications to have proven that excluding any ordered matching yields a $\chi$-bounded class. 
\begin{conjecture}\label{conj:bartosz}
    Let $(M,\prec)$ be an ordered graph with maximum degree $1$. Then the class of $(M,\prec)$-free ordered graphs is $\chi$-bounded. 
\end{conjecture}

By Property~\ref{prop:tournaments_orderedgraphs}, we have:
\begin{observation}\label{lem:matching}
    If Conjecture~\ref{conj:bartosz} holds, then any tournament that admits a backedge graph of maximum degree $1$ is $\dic$-binding.
\end{observation}

One of the first graphs shown to be $\chi$-binding were paths~\cite{Gya87}. Hence, a natural tournament for which one might wonder whether it is $\dic$-binding or not is the following. Let $\TP_n$ be the tournament on $n$ vertices obtained from the transitive tournament $TT_n$ by reversing the direction of each arc of the unique Hamiltonian path $v_1v_2\ldots v_n$. Now consider $\prec$ to be the ordering $v_2 \prec v_1 \prec v_4 \prec v_3 \prec \ldots \prec v_{2p} \prec v_{2p-1} \prec \ldots \prec v_n \prec v_{n-1}$ (assuming for simplicity that $n$ is even, otherwise we end with $v_n$). It is easy to observe that the backedge graph of $\TP_n$ with respect to $\prec$ has maximum degree $1$. Hence, by ~\ref{prop:tournaments_orderedgraphs}, we have:

\begin{observation}\label{lem:GS_holds_for_Pk}
If Conjecture~\ref{conj:bartosz} holds, then for every $k \geq 1$, $\TP_k$ is $\dic$-binding.
\end{observation}


\subsection
[Relation with the Erd\H{o}s-Hajnal property and the BIG-BIG conjecture]
{Relation with the Erd\H{o}s-Hajnal property and the $BIG \Ra BIG$ conjecture}\label{sec:erdos_hajnal}

A tournament $H$ has the \emph{Erd\H{o}s-Hajnal property} if there exists an integer $c$ such that for every $H$-free tournaments $T$, $T$ contains a transitive tournament on  $|T|^{c}$ vertices. It was proven in \cite{alon2001ramsey} that the famous \Erd-\Haj conjecture on undirected graphs is equivalent to the conjecture saying that every tournament has the \Erd-\Haj property.

\begin{theorem}
If $H$ is a polynomially $\dic$-binding tournament, then $H$ has the Erd\H{o}s-Hajnal property. 
\end{theorem}

\begin{proof}
Let $H$ be a tournament such that the class of $H$-free tournaments is polynomially $\dic$-binding. Then there is an integer $c$ such that for every $H$-free tournaments $T$, $\dic(T) \leq \diomega(T)^c$. Let us prove that $H$ has the Erd\H{o}s-Hajnal property (for the constant $\frac{1}{1+c}$).

Let $T$ be an $H$-free tournament. If $\diomega(T) \geq n^{\frac{1}{1+c}}$, then $T$ contains a transitive tournament of size $n^{\frac{1}{1+c}}$ and we are done.  So assume that $\diomega(T) \leq n^{\frac{1}{1+c}}$. Then $\dic(T) \leq n^{\frac{c}{1+c}}$, i.e. $T$ can be partitioned into $n^{\frac{c}{1+c}}$ transitive tournaments. Hence, by the pigeonhole principle, $T$ contains a transitive tournament on at least $\frac{n}{n^{\frac{c}{1+c}}}=n^{\frac{1}{1+c}}$ vertices. 
\end{proof}



We say that a class of tournaments $\mc T$ has the \emph{$BIG \Ra BIG$ property} if there exists a function $f$ such that, for every $T \in \mc T$, if $\dic(T) \geq f(t)$, then $T$ contains two disjoint subtournaments $A$ and $B$ such that $\dic(A), \dic(B) \geq t$ and  $A \Ra B$. 
In~\cite{NSS23}, the following beautiful conjecture is proposed:
\begin{conjecture}[$BIG \Ra BIG$ Conjecture\ \cite{NSS23}]
    The class of all tournaments has the $BIG \Ra BIG$ property. 
\end{conjecture}


Nguyen, Scott and Seymour proved in \cite{elzahar} that the $BIG \Ra BIG$ Conjecture implies the Erd\H{o}s-El-Zahar conjecture, which states that there exists a function $f$ such that, for every integer $c$, every graph $G$ with $\chi(G) \geq f(\omega(G), c)$ contains two disjoint subgraphs $A$ and $B$ such that $\chi(A), \chi(B) \geq c$ and  there is no edge between $A$ and $B$. Klingelhöfer and Newman~\cite{KN23} showed recently the other direction, that is  the Erd\H{o}s-El-Zahar conjecture implies the $BIG \Ra BIG$ Conjecture. 
 To prove it, they first prove the following beautiful theorem.  
Given an oriented graph $D$, we denote by $\alpha(D)$ the size of a maximum independent set of $D$ and by $\omega(D)$ the clique number of the underlying graph of $D$. Recall that, given an arc $uv$, $N(uv) = N^+(v) \cap N^-(u)$. 
\begin{theorem}[\cite{KN23}]\label{thm:arc_local_to_global_alpha}
There exists a function $\lambda$ such that, for every integer $t$, if $D$ is an oriented graph such that for every $a \in A(T)$, $\dic(D[N(a)]) \leq t$, then $\dic(D) \leq \lambda(t, \alpha(D))$. 
\end{theorem}

Applying the  same method that  Klingelhöfer and Newman used to prove that the Erd\H{o}s-El-Zahar conjecture implies the $BIG \Ra BIG$ Conjecture, we can prove the following.  

\begin{theorem}
Let $\mathcal T$ be a $\dic$-bounded class of tournaments. Then $\mathcal T$ has the $BIG \Ra BIG$ property
\end{theorem}

\begin{proof}
Let $\ell$ be a function defined inductively as follows: $\ell(1) = 1$ and for every $t \geq 1$, $\ell(t+1)=(t+1)+\binom{t+1}{2} \ell(t)$. 
Given a tournament $T$, we say that a subtournament $X$ of $T$ is a \emph{$t$-cluster} if $\dic(X) \geq t$ and $|X| \leq \ell(t)$.  

Let $c$ be an integer. 
Let $T \in \mathcal T$ such that $T$ does not contain  two disjoint subtournaments $A$ and $B$ such that  $A \Ra B$ and  $\dic(A),\dic(B) \geq c$. We want to prove that the dichromatic number of $T$ is bounded by a function of $c$.

\begin{claim}
   For every $t \leq 2c$, if $T$ contains no $t$-cluster,  then $\dic(T)$ is bounded by a function of $c$ and $t$. 
\end{claim}
\begin{proofclaim}
Since a $1$-cluster is a vertex, the result trivially holds for $t=1$. 
Now assume it holds for $t <2c$, and let us prove it for $t+1$. So assume $T$ has no $(t+1)$-cluster, and say that an arc $a$ is \emph{heavy} if $N(a)$ contains a $t$-cluster, and it is \emph{light} otherwise (we recall that if $xy$ is an arc, $N(xy)$ denotes the set of vertices $z$ such that  $yz\in A$ and $zx\in A$). 
Let $T_h$ be the oriented graph induced by the heavy arcs, and $T_{\ell}$ the oriented graphs induced by the light arcs. 

We first claim that  $\omega(T_h) \leq t$ (recall that $\omega(T_h)$ denotes the clique number of the underlying graph of $T_h$). Assume by contradiction there exists $K$ of size $t+1$ inducing a tournament in $T_{h}$. For every arc $a$ with both endvertices in $K$, $a$ is heavy so there exists $C_a$ a $t$-cluster included in $N(a)$. Let $X$ be the subtournament of $T$ induced by the union of $K$ and all such sets $C_a$. The number of its vertices is at most $(t+1)+\binom{t+1}{2}l(t)=l(t+1)$. If $X$ admits a dicolouring with at most $t$ colours then there must be two vertices $x,y$ in $K$ that get the same colour (because $K$ has size $t+1$), but then this colour cannot appear in $C_{xy}$ (for it would create a monochromatic $\vec C_3$), which contradicts the fact that $C_{xy}$ has dichromatic number at least $t$. Hence $\dic(X)\geq t+1$ and so $X$ is a $(t+1)$-cluster which contradicts our hypothesis. Hence we have proven that $\omega(T_h) \leq t$, which implies that $\alpha(T_{\ell}) \leq t$. 

By induction, since for every light arc $a$, $N(a)$ contains no $t$-cluster, we have that $\dic(N^+(a))$ is bounded by a function of $t$ and $c$, and thus by a function of $c$ since $t \leq 2c$. 
Now, by Theorem~\ref{thm:arc_local_to_global_alpha}, $\dic(T_{\ell})$ is also bounded, say by $k$, where $k$ is a function of $t$ and $\alpha(T_{\ell}) \leq t$, so again a function of $c$.  
Let $(S_1, \dots, S_k)$ be a dicolouring of $T_{\ell}$, i.e. $T_{\ell}[S_i]$ is acyclic for $i=1, \dots, k$. Then, for $i=1, \dots, k$, there is an ordering $\prec_i$ of $S_i$ such that  all backward arcs of $(T[S_i], \prec_i)$ are heavy. Hence $\diomega(T[S_i]) \leq t$ (because $\omega(T_h) \leq t$) and since $T[S_i]$ is $H$-free and $H$  is $\dic$-binding, $\dic(T[S_i])$ is bounded by a function of $\diomega(T[S_i]) \leq t \leq 2c$. Altogether, it implies that $\dic(T)$ is also bounded by a function of $c$. This concludes the proof by induction. 
\end{proofclaim}

We can now conclude. Either $T$ contains no $2c$-cluster and we conclude by what precedes, or $T$ contains a $2c$-cluster $X$. 
Partition $V(T) \setminus V(X)$ with respect to their adjacency to $X$. This gives a partition of $V(T) \setminus X$ into at most $2^{|X|} \leq 2^{\ell(2c)}$ parts. 
Assume by contradiction that one of these parts, let it be $A$, has dichromatic number at least $c$. Call $B^+$ (resp. $B^-$) the subset of $X$ such that $A \Ra B^+$ (resp $B^- \Ra A$). Since $\dic(X) \geq 2c$, one of $B^+$ or $B^-$ has dichromatic number of at least $c$, and we get a contradiction with the assumption on $T$. So every such part $A$ has dichromatic number at most $c$ and hence $\dic(T)$ is bounded by $\ell(2c)+2^{\ell(2c)}c$.

\end{proof}




\section{Local to Global - Links with domination number}\label{sec:cluster}


Informally, given a digraph parameter $\gamma$, a $\gamma$-cluster of a tournament $T$ is a subtournament $X$ of $T$ of bounded size with large $\gamma$. In this section, we investigate for which parameters $\gamma_1$ and $\gamma_2$ we have that, for all tournaments $T$ with sufficiently large $\gamma_1$, $T$ has a $\gamma_2$-cluster. 

We say that \emph{large $\gamma_1$ implies a $\gamma_2$-cluster} if there exists two functions $f$ and $\ell$ such that for every integer $k$, if $\gamma_1(T) \geq f(k)$, then $T$ contains a subtournament $X$ such that $\gamma_2(X) \geq k$ and $|X| \leq \ell(k)$. We review what is known on this topic and propose some new conjectures.

Such properties were first studied by Thomassé, Le, Harutyunyan and Wu in~\cite{HLTW19}:

\begin{theorem}[Large $dom$ implies a $\dic$-cluster, \cite{HLTW19}]\label{thm:dom_chi_cluster}
    There exist two functions $f$ and $\ell$ such that, for every integer $k$, every tournament $T$ with $\dom(T) \geq f(k)$ contains a subtournament $X$ with $|X| \leq \ell(k)$ and $\dic(X) \geq k$. 
\end{theorem} 

In the same paper, Thomassé, Le, Harutyunyan and Wu conjectured the following: 

\begin{conjecture}[Large $dom$ implies a $dom$-cluster, \cite{HLTW19}]\label{conj:dom_dom_cluster}
    There exist two functions $f$ and $\ell$ such that, for every integer $k$, every tournament $T$ with $\dom(T) \geq f(k)$ contains a subtournament $X$ with $|X| \leq \ell(k)$ and $\dom(X) \geq k$. 
\end{conjecture}

Since $\dom(T)  \leq \diomega(T) \leq \dic(T)$ (see Property~\ref{prop:dom<omega}), the following is stronger than Theorem~\ref{thm:dom_chi_cluster} but weaker than Conjecture~\ref{conj:dom_dom_cluster}. 
\begin{conjecture}[Large $dom$ implies a $\diomega$-cluster] \label{conj:dom_omega_cluster}
    There exist two functions $f$ and $\ell$ such that, for every integer $k$, every tournament $T$ with $\dom(T) \geq f(k)$ contains a subtournament $X$ with $|X| \leq \ell(k)$ and $\diomega(X) \geq k$. 
\end{conjecture}


A tournament $R$ is a \emph{rebel} if $R$-free tournaments  have bounded domination number. 
An undirected graph $G$ is a \emph{comparability graph} if there is a poset such that the  vertices of $G$ are the elements of the poset, and two vertices are adjacent if and only if the corresponding elements are comparable in the poset. 
A tournament is a \emph{poset tournament} if it admits a backedge graph that is a comparability graph. Chudnovsky, Kim, Liu, Seymour and Thomassé proved in \cite{CKLST18} that every rebel is a poset tournament, but the converse remains open.

\begin{conjecture}[Chudnovsky, Kim, Liu, Seymour and Thomassé]\label{conj:rebel}
    Every poset tournament is a rebel.
\end{conjecture}

In particular, the $S_k$ defined at the end of Section~\ref{sec:subst_S_k} are poset  tournaments. Let us prove it by induction on $k$. $S_1=TT_1$ is clearly a poset tournament. Now consider $S_{k+1}$, and set $V(S_{k+1}) = \{x\} \cup A \cup B$ where $A$ and $B$ induce $S_{k}$ and $x \Ra A \Ra B \Ra x$. Let $\prec$ be an ordering such that $x\prec A \prec B$ and such that $A^{\prec}$ and $B^{\prec}$ are comparability graphs. Then it is clear that $S_{k+1}^{\prec}$ is a comparability graph.   
Moreover, we proved at the very end of Section~\ref{sec:subst_S_k} that  $\diomega(S_k) \geq log_9(k)$. Hence, if one can prove that, for every integer $k$, $S_k$ is a rebel, then Conjecture~\ref{conj:dom_omega_cluster} holds. 
Indeed, if $S_k$ is a rebel, then there is a function $f$ such that, for every tournament $T$, $dom(T) \geq f(k)$ implies that $T$ contains a subtournament $X$ isomorphic to $S_{9^k}$, and since $\diomega(S_{9^k}) \geq k$, taking $\ell(k) = |V(S_{9^k})|$ works.  
In particular: 

\begin{theorem}
    Conjecture~\ref{conj:rebel} implies Conjecture~\ref{conj:dom_omega_cluster}.
\end{theorem}

Thomassé, Le, Harutyunyan and Wu applied Theorem~\ref{thm:dom_chi_cluster} to prove Theorem~\ref{thm:local_to_global}, which can be seen as a local to global theorem about dichromatic number. The analogue of Theorem~\ref{thm:local_to_global} for clique number  is the following conjecture.

\begin{conjecture}\label{conj:localtoglobalomega}
    There exists a function $g$ such that, for every integer $t$, if $T$ is a tournament such that for every $v \in V(T)$, $\diomega(N^+(v)) \leq t$, then $\diomega(T) \leq g(t)$. 
\end{conjecture}
The analogue of Theorem \ref{thm:dom_chi_cluster} for clique number is Conjecture \ref{conj:dom_omega_cluster} and indeed we have the following implication.
\begin{theorem}
    Conjecture~\ref{conj:dom_omega_cluster} implies Conjecture~\ref{conj:localtoglobalomega}
\end{theorem}
\begin{proof}
    Let $T$ be a tournament and $t \in \mathbb{N}$ such that for every vertex $v \in V(T)$, $\diomega(N^+(v)) \leq t$. Let $f$ and $\ell$ be the functions given by Conjecture~\ref{conj:dom_omega_cluster}. We will prove that $\diomega(T) \leq \max(tf(t+1), t\ell(t+1) + t+1)$.

    If $\dom(T) < f(t+1)$, then, since $\diomega(N^+(v)) \leq t$ for every $v \in V(T)$, we have $\diomega(T) < tf(t+1)$. 
    So we may assume that $\dom(T) \geq f(t+1)$ and thus
    $T$ has a subtournament $X$ such that $|X| \leq \ell(t+1)$ and $\diomega(X) =  t+1$.  Hence, since for every $v \in V(T)$, $\diomega(N^+(v)) \leq t$, we have $X \subsetneq N^+(v)$. So $X$ is a dominating set of $T$. Hence, $V(T) = X \cup \cup_{x \in X}N^+(x)$, and since for every vertex $x$, $\diomega(N^+(x)) \leq t$, we get that $\diomega(T) \leq t|X|+ \diomega(X) \leq t|X|+t+ 1 \leq t\ell(t+1) +t+1$
\end{proof}

Since $\dom(T) \leq \diomega(T)$ for every tournament $T$, the following conjecture implies Conjecture~\ref{conj:dom_omega_cluster} and would give a  natural property of the clique number of tournaments. 
\begin{conjecture}[Large $\diomega$ implies an $\diomega$-cluster]\label{conj:omega_omega_cluster}
     There exists two functions $f$ and $\ell$ such that, for every integer $k$, every tournament $T$ with $\diomega(T) \geq f(k)$ contains a subtournament $X$ with $|X| \leq \ell(k)$ and $\diomega(X) \geq k$.  
\end{conjecture}

We believe (or maybe only hope) that the above conjecture is true, and actually we were not even able to answer the following stronger form of it, where $f$ is taken to be the identity: 
\begin{question}\label{question:super_omega_cluster}
  Is there a function $\ell$ such that, for every tournament $T$, if $\diomega(T) \geq  k$, then $T$ has a subtournament $A$ such that $|A| \leq \ell(k)$ and $\diomega(A) \geq k$.  
\end{question}

 Let us say that a tournament $T$ is $k$-$\diomega$-critical if $\diomega(T) = k$ and for every $v \in V(T)$, $\diomega(T-v) = k-1$. Observe that the only $1$-$\diomega$-critical tournament if the one vertex tournament, and the only $2$-$\diomega$-critical tournament is $\vec C_3$. 
 \begin{conjecture}\label{conj:critique}
     For every integer $k \geq 3$, there are an infinite number of  $k$-$\diomega$-critical tournaments. 
 \end{conjecture}

 Observe that if Conjecture~\ref{conj:critique} is true (resp. false), then it answers Question~\ref{question:super_omega_cluster} in the negative (resp. in the positive).
 \medskip

It is also open if large $\diomega$ implies a $\dic$-cluster (resp. a $dom$-cluster). 
\medskip 

On the negative side, Thomassé, Le, Harutyunyan and Wu~\cite{HLTW19} proved the following:
\begin{theorem}[Large $\dic$ does not imply a $\dic$-cluster, \cite{HLTW19}]
    For every integer $K, \ell$, there exists a tournament $T$ such that $\dic(T) \geq K$, and all subtournaments $X$ of $T$ on at most $\ell$ vertices are $2$-dicolourable. 
\end{theorem}


\section{Conclusion and future directions}\label{sec:conclusion}


We did not give much thought yet to the clique number of digraphs. On this matter, it is to be noted that Theorem~\ref{eq:bounddic} does not hold for digraphs. Indeed, take a triangle-free (undirected) graph $G$ with large chromatic number, say $\chi(G) = k$. 
Let $\prec$ be an ordering of $V(G)$ and let $\vec G$ be the digraph obtained from $G$ by orienting each edge from right to left with respect to $\prec$. Then $\chi(T^{\prec})/\omega(T^{\prec}) = \chi(G) / \omega(G) = k/2 > \dic(\vec G) = 1$. 
Since most of the results we have proven for tournaments rely on Theorem~\ref{eq:bounddic}, we believe that studying the clique number of digraphs may not be as fruitful as studying the clique number of tournaments.

 On the positive side, we can extend Theorem~\ref{eq:bounddic}  to digraphs with bounded independence number, where the \emph{independence number} of a digraph $D$ is the size of a maximum independent set of (the underlying graph) of $D$. Since tournaments are precisely oriented graphs with independence number $1$, it is indeed a generalisation of Theorem~\ref{eq:bounddic}. 
 Note that many results on tournaments extend to digraphs with bounded independence number. See for example~\cite{HLNT19,KN23}. 

We denote by $\alpha(D)$ the size of a maximum independent set in $D$ and by $R(i,j)$ the smallest integer such that every graph on  $R(i,j)$ vertices contains either a clique of size $i$, or an independent set of size $j$. $R(i,j)$ exists for all  integers $i,j$ by Ramsey's Theorem. 

\begin{theorem}[\cite{NSS23}]\label{eq:bounddic_alpha}
For any digraph $D$ and ordering $\prec$ of $V(D)$, we have:
\[ 
\frac{\chi(D^{\prec})}{R(\omega(D^{\prec})+1, \alpha(D)+1)} \leq \dic(D) \leq \chi(D^{\prec})
\]
\end{theorem}

\begin{proof}
Let $D$ be a digraph and $\prec$ an ordering of $V(D)$. 
Set $\omega = \omega(D^{\prec})$ and $\alpha = \alpha(D)$. 
Let $X \subseteq V(D)$ such that $D[X]$ is acyclic. 
Let us first prove that   $\chi(D^{\prec}[X]) \leq R(\omega+1,\alpha+1) $. 

Let $\varphi : X \to \mathbb{N}$ be such that $\varphi(x)$ is the number of vertices of a longest $\prec$-decreasing path in $D^{\prec}[X]$ finishing in $x$. We claim that $\varphi$ is a $R(\omega+1,\alpha+1) $-colouring of $D^{\prec}$. 
Let $u,v \in X$ with $u\prec v$ and $uv \in E(D^{\prec})$. Then $\varphi(u) \geq \varphi(v) + 1$, so $\varphi$ is a proper colouring of $D^{\prec}[X]$. 

Suppose for contradiction that $\varphi$ uses more than $R(\omega+1,\alpha+1) $ colours. Then there is a $\prec$-decreasing path $P$ of size $R(\omega+1,\alpha+1)  + 1$. Note that since $D[V(P)]$ is acyclic, every arc of $D[V(P)]$ corresponds to an edge of $D^{\prec}[V(P)]$ (i.e. for every $x,y \in V(P)$, if $xy \in A(D)$, then $y \prec x$ and thus $xy \in E(D^{\prec})$). Hence, $\alpha(D^{\prec}[V(P)] = \alpha(D[V(P)]) \leq \alpha$. 
Now, since $|V(P)| > R(\omega+1,\alpha+1) $ and $\alpha(D^{\prec}[V(P)]) \leq \alpha$, we get that  $\omega(D^{\prec}[V(P)]) \geq \omega + 1$, a contradiction. 

Now, $D$ can be partitioned into $\dic(D)$ acyclic subdigraphs, and for each of these acyclic subdigraphs $X$, $D^{\prec}[V(X)]$ can be coloured with $R(\omega+1,\alpha+1)$ colours. We thus have  $\chi(D^{\prec}) \leq \dic(D) R(\omega(D^{\prec})+1, \alpha(D)+1) $. 
\end{proof}

Using the above inequation, we believe most results of the paper can be adapted to classes of digraphs with bounded independence number. 

Despite what we have just explained, we believe that Theorem~\ref{thm:stable_by_subst} can be generalized to classes of digraphs.

\begin{conjecture}\label{conj:closed_subst}
    If a class of digraphs $\mc C$ is $\dic$-bounded, then so is its closure under substitution. 
\end{conjecture}

An obvious topic that we did not investigate is the complexity of computing the clique number. Nguyen, Scott and Seymour ask for the complexity of computing the clique number of a tournament in Section 9 of \cite{NSS23}. 
Between the time we submitted the paper and the time we receive the reviews, the second author proved \cite{A24} that for any $k \geq 3$ fixed, deciding if the clique number of a tournament is at most $k$ is NP-complete. The question for $k=2$ remains open.

    
    


\subsubsection*{Acknowledgement}
This research was partially supported by the ANR project DAGDigDec (JCJC)   ANR-21-CE48-0012 and by the group Casino/ENS Chair on Algorithmics and Machine Learning.

\end{document}